\documentclass{amsart}
\usepackage{amsmath,amsfonts, amssymb,ifthen}

\newcommand{\integers}{\mathbb Z}

\newcommand{\cd}{\text{cd}}

\theoremstyle{definition}
\newtheorem{thm}{Theorem}%[section]
\newtheorem{prob}{Problem}

\theoremstyle{remark}
\newtheorem{claim}{Claim}

\newcommand{\showcomments}{yes}
\renewcommand{\showcomments}{no}

\newsavebox{\commentbox}
{\ifthenelse{\equal{\showcomments}{yes}}%
{\footnotemark
        \begin{lrbox}{\commentbox}
        \begin{minipage}[t]{1in}\raggedright\sffamily\tiny
        \footnotemark[\arabic{footnote}]}
{\begin{lrbox}{\commentbox}}}%
{\ifthenelse{\equal{\showcomments}{yes}}%
{\end{minipage}\end{lrbox}\marginpar{\usebox{\commentbox}}}
{\end{lrbox}}}

\title{A curiously cubulated group}
\author{Kasia Jankiewicz}
\address{Dept. of Math. \\ University of Chicago, Chicago, Illinois, 60637}
\email{kasia@math.uchicago.edu}

\author[D.~T.~Wise]{Daniel T. Wise}
           \address{Dept. of Math. \& Stats.\\
                    McGill Univ. \\
                    Montreal, QC, Canada H3A 0B9 }
           \email{wise@math.mcgill.ca}

\begin{document}

\begin{abstract}
We construct a finitely generated 2-dimensional group that acts properly on a locally finite CAT(0) cube complex
but does not act properly on a finite dimensional CAT(0) cube complex.
\end{abstract}

\maketitle
%\vspace{-5mm}
Among its  extraordinary properties, 
Thompson's group $F$  acts freely and metrically properly on a CAT(0) cube complex \cite{Farley2003},
but $F$ does not act properly on a finite dimensional CAT(0) cube complex,
 since $F$ contains a copy of $\integers^\infty$. The purpose of this note is to produce a f.g.\ group $G$ with
the same property, but with $\cd(G)<\infty$.
We remark that there are tubular groups (e.g.\ Gersten's group) 
that act freely on infinite dimensional cube complexes,
but cannot act metrically properly \cite{WiseGerstenRevisited}, since Gersten's result applies in a cubical case \cite{Gersten94, HaglundSemiSimple, WoodhouseAxis}. 
Also, the wreath product $H\wr F_2$ with $H\neq 1$ finite, 
acts metrically properly on an infinite dimensional CAT(0) cube complex, 
but does not act metrically properly on any finite dimensional CAT(0) cube complex \cite{CornulierStadlerValette12}. 
%This groups is not torsion-free, but admits a $2$-dimensional model for the classifying space of proper actions.

\begin{thm}\label{thm:main}
There exists a f.g.\ group of cohomological dimension $2$ that acts freely (and thus metrically properly)  on a locally-finite CAT(0) cube complex but does not act 
freely on a finite dimensional CAT(0) cube complex.
\end{thm}

\begin{proof} 
For each $n\geq 1$ let $A_n$ be the $2$-complex of a $2$-generator $C'(\frac16)$ group that does not act properly on any $n$-dimensional CAT(0) cube complex. Such groups were constructed in \cite{JankiewiczCubicalDimension}. Moreover, we can choose $A_n$ so that its $2$-cells have boundary cycles of length $>12$.

Let $X_n = A_n\vee B_n$ where $B_n$ is a bouquet of two circles. Let $X$ be the union of a ray $r =[0,\infty]$ and $\displaystyle \sqcup_{n\geq 0} X_n$ where the basepoint of $X_n$ is identified with the vertex $n\in r$. Let $t_n$ denote the oriented edge $[n-1,n]$ of $r$. Let the basepoint of $X$ be the vertex $0 \in r$. We have $\pi_1 X = *_{n=0}^{\infty}\pi_1 X_n$.  
 Let $x_{n1},x_{n2}$ be closed paths in $X_n$ representing generators of $\pi_1 A_n$, and let $x_{n3},x_{n4}$ be closed paths representing generators of $\pi_1B_n$.

Consider the $2$-complex $Y$ obtained from $X$ by adding $2$-cells as follows:
for each $n\geq 1$ and  $i\in\{1,2,3,4\}$ add a $2$-cell $C_{ni}$ with boundary path
\[
t_nx_{ni}t_n^{-1}\gamma_{ni}
\]
where $\gamma_{ni}$ is a closed path in $X_{n-1}$ of the form 
\[
\gamma_{ni} =
\beta_{ni}^1\alpha_{ni}\beta_{ni}^2\cdots\alpha_{ni}\beta_{ni}^m
\]
 where $m\geq 12$, and where
 \begin{itemize}
 \item $\alpha_{ni} = x_{(n-1)i}$ for $i=1,2$, and $\alpha_{ni} = (x_{(n-1)(i-2)})^2$ for $i=3,4$;\\ (note that $\alpha_{ni}\in A_{n-1}$),
 \item $\{\beta_{ni}^j : 1\leq i\leq 4,\, 1\leq j\leq m\}$ 
is a collection of distinct closed paths of the same length in $B_{n-1}$.
 \end{itemize}

\begin{claim}The complex $Y$ is a $C'(\frac16)$-complex without periodic 2-cell attaching maps. Consequently, $\cd(\pi_1 Y)=2$.\end{claim}
As $X$ is a $C'(\frac16)$-complex, it suffices to verify the $C'(\frac16)$ condition for:
\begin{itemize}
\item a $2$-cell $C_{ni}$ and a $2$-cell $C$ of $X$,
\item a $2$-cell $C_{ni}$ and a $2$-cell $C_{n'i'}$.
\end{itemize}

If there is nontrivial piece between $C_{ni}$ and $C$, then $C$ lies in $A_n$ or $A_{n-1}$ and the piece is the first or second power of one of $\{x_{n1}, x_{n2}, x_{(n-1)1},x_{(n-1)2}\}$. 

The cells $C_{ni}$ and $C_{n'i'}$ can only overlap if $|n-n'|\leq 1$.
If $n' = n+1$, then 
the only piece between $C_{ni}$ and $C_{n'i'}$ is of the form $x_{ni}$.
In all the cases considered so far, pieces have length at most $2$, so the $C'(\frac16)$ condition is satisfied.

If  $n=n'$ and  $i - i'$ is odd, then any nontrivial piece is either $t_n$ or is contained in some $\beta_{ni}^j$. Suppose $n=n'$ and $i-i'$ is even. Since $\{\beta_{ni}^j\}$ are distinct, the maximal pieces are strictly contained in $\beta_{ni}^{j}\alpha_{ni}\beta_{ni}^{j+1}$, $t_n^{-1}\beta_{ni}^1$, or $\beta_{ni}^m t_n$. In these cases $C'(\frac16)$ condition holds
as $m\geq 12$.

Any $C'(\frac16)$-complex without periodic 2-cell attaching maps is aspherical (see e.g.\ \cite{LS77} and \cite{ChiswellCollinsHuebschmann81})
so $\cd (\pi_1Y) =2$.

\begin{claim} The group $\pi_1 Y$ is f.g.
\end{claim}
$\pi_1 Y$ is a quotient of $\pi_1 X$, so $\pi_1Y$ is generated by images of generators of $\pi_1 X$, i.e.\ the elements represented by $\bigcup_{n=0}^{\infty}\{t_1\cdots t_{n}x_{ni}t_{n}^{-1}\cdots t_{1}^{-1}\}_{i=1}^{4}$. By construction, $t_1\cdots t_{n}x_{ni}t_{n}^{-1}\cdots t_{1}^{-1}$ is homotopic in $Y$ to a closed path contained in a finite subcomplex $X_{[0,n-1]} = [0,n-1]\cup \bigcup_{i=0}^{n-1} X_i$ of $X$.  Specifically, 
\[[t_1\cdots t_{n}x_{ni}t_{n}^{-1}\cdots t_{1}^{-1}] = [t_1\cdots t_{n-1}\gamma_{ni}^{-1}t_{n-1}^{-1}\cdots t_{1}^{-1}] \,\,\text{ in }\pi_1Y.\] 
 Consequently, $\pi_1 Y$ is generated by $[x_{01}], [x_{02}], [x_{03}],[x_{04}]$.

\begin{claim}The complex $Y$ is locally finite.\end{claim}

Indeed, each $X_n$ is compact and intersects finitely many additional cells:\\ $\{t_{m}, C_{mi}: m\in\{n, n+1\}, i\in\{1,2,3,4\}\}$.

\begin{claim} The group $\pi_1 Y$ acts freely on a locally finite CAT(0) cube complex.
\end{claim}
As $Y$ is locally finite, its dual cube complex is locally finite~\cite[Thm~10.5]{WiseSmallCanCube04}. 
Since $Y$ is aspherical, $\pi_1Y$ is torsion-free.
Hence the action of $\pi_1 Y$ on the dual cube complex is free~\cite[Cor~10.7]{WiseSmallCanCube04}.

\begin{claim}The groups $\pi_1 A_n$ embed in $\pi_1 Y$. Consequently, $\pi_1 Y$ does not act properly on any finite dimensional CAT(0) cube complex.\end{claim}
$\pi_1A_n\hookrightarrow \pi_1 Y$ since $A_n\hookrightarrow Y$ has  \emph{no missing shells} \cite[Thm~13.3]{WiseCBMS2012}. As $\pi_1A_n$ does not act properly on an $n$-dimensional CAT(0) cube complex, $\pi_1Y$ does not.
\end{proof}

Groups with finite $C'(\frac 16)$ presentations are hyperbolic. However, the complex $Y$ in Theorem~\ref{thm:main} has infinitely many $2$-cells. We close with the following problem, which is open, even under the assumption that $G$ acts freely on a finite dimensional CAT(0) cube complex.
\begin{prob}
Let $G$ be a hyperbolic group.
Suppose $G$ acts freely on a CAT(0) cube complex.
Does $G$ act freely and cocompactly on a CAT(0) cube complex?
\end{prob}

\subsection*{Acknowledgements} We are grateful to the referee for helpful comments and corrections.
%%%%%%%%%%%%%%%%%%%%%%%%%%%%%%%%%%%%%%%%%%%%%%%%%%%%%%%%%%%%%%%%%%%%%%%%
%%                  BIBLIOGRAPHY
%%%%%%%%%%%%%%%%%%%%%%%%%%%%%%%%%%%%%%%%%%%%%%%%%%%%%%%%%%%%%%%%%%%%%%%%
\bibliographystyle{alpha}
\bibliography{wise}
%\bibliography{wise}

\end{document}